\documentclass[12pt]{article}

\usepackage[utf8x]{inputenc}
\usepackage[english]{babel}
\usepackage{amssymb, amsmath, amsthm, verbatim, cite, lscape, longtable, graphicx, wrapfig, tikz}

\def\C{{\mathbb C}}
\def\T{{\mathbb T}}

\newtheorem{lemma}{Lemma}
\newtheorem{theorem}[lemma]{Theorem}
\newtheorem{example}[lemma]{Example}

\newtheorem{definition}[lemma]{Definition}
\newtheorem{remark}[lemma]{Remark}

\title{A switching method for constructing \\cospectral gain graphs}

\author{
Aida Abiad
\thanks{\texttt{a.abiad.monge@tue.nl}, Department of Mathematics and Computer Science, Eindhoven University of Technology, The Netherlands} \thanks{Department of Mathematics: Analysis, Logic and Discrete Mathematics, Ghent University, Belgium} \thanks{Department of Mathematics and Data Science, Vrije Universiteit Brussel, Belgium}
\and
Francesco Belardo
\thanks{\texttt{fbelardo@unina.it}, Department of Mathematics and Applications R. Caccioppoli, University of Naples Federico II, Italy} 
\and
Antonina P. Khramova
\thanks{\texttt{a.khramova@tue.nl}, Department of Mathematics and Computer Science, Eindhoven University of Technology, The Netherlands} 
}

\date{}

\begin{document}
\maketitle

\begin{abstract}
A gain graph over a group $G$, also referred to as $G$-gain graph, is a graph where an element of a group $G$, called gain, is assigned to
each oriented edge, in such a way that the inverse element is associated with the opposite
orientation. Gain graphs can be regarded as a generalization of signed graphs, among others. In this work, we show a new switching method to construct cospectral gain graphs. Some previous methods known for graph cospectrality follow as a corollary of our results.
\end{abstract}

%%%%%%%%%%%%%%%%%%%%%%%%%%%%%%%%%%%%%%%%%%%%%%%%%
\section{Introduction}
%%%%%%%%%%%%%%%%%%%%%%%%%%%%%%%%%%%%%%%%%%%%%%%%%

%A group representation approach to balance of gain graphs \cite{CDD2021}

%A $G$-gain graph is a graph where an element of a group $G$, called gain, is assigned to each oriented edge, in such a way that the inverse element is associated with the opposite orientation.
Gain graphs, which can be regarded as a generalization of signed graphs \cite{CDD2021}  where the group is $\{-1, 1\}$, have been extensively studied, see the survey paper \cite{zaslavskybib}. Gain graphs can be also considered as particular cases of biased graphs \cite{z1989} and are related to voltage graphs \cite{GT1977}. Among the key notions concerning signed and gain graphs, there are switching equivalence and balance, see for instance \cite{H1953,z1989}.

The spectrum of complex unit gain graphs and signed graphs with respect to the identical representation, both of which are special cases of gain graphs and the more general $G$-spectrum, has been considered in the literature, see \cite{zaslavskybib}. Reff \cite{R2012} introduced complex unit gain graphs and investigated their spectral properties by extending some fundamental results from spectral graph theory. For signed graphs, characterizations of some families have been done with respect to the identical representation, such as~\cite{ABDN2018,AHMM2018,BP2015}. Other cases related to gain graphs have been considered, including signed directed graphs~\cite{BBCRS2022} and quaternion unit gain graphs \cite{WD2022}. For a general group $G$, it has been shown that any $G$-gain graph with a cycle as its underlying graph is determined by its $G$-spectrum~\cite{CD2022}.

Less is known regarding cospectrality of gain graphs. The well-known switching method by Godsil and McKay \cite{gm1982} (GM-switching), and the more recent switching by Wang, Qiu and Hu \cite{WQH2019,QJW2020} (WQH-switching), are operations on graphs that do not change the spectrum of the adjacency matrix. The GM-switching has been extended to signed graphs \cite{Belardo1} and complex unit graphs \cite{Belardo2}, and more recently, also to gain graphs over an arbitrary group \cite{CDS2022}. The WQH-switching has also been extended, in its simpler form, to signed graphs \cite{Belardo1} and complex unit graphs \cite{Belardo2} under the name Modified Godsil-McKay switching. In this paper we show that both results from \cite{Belardo1,Belardo2} hold in a more general form. In particular, we present a new method to construct $G$-gain graphs over an arbitrary group $G$ that are cospectral with respect to the $G$-spectrum, which is independent of the choice of the particular representation of $G$. We will also investigate how this new switching behaves when considering a particular representation of~$G$. Our results provide new insights into the spectral theory of gain graphs with respect to $G$-cospectrality, which to the best of our knowledge, has only been investigated by Cavaleri, Donno and Spessato \cite{CD2022,CDS2022}.

%%%%%%%%%%%%%%%%%%%%%%%%%%%%%%%%%%%%%%%%%%%%%%%%%
\section{Preliminaries}
%%%%%%%%%%%%%%%%%%%%%%%%%%%%%%%%%%%%%%%%%%%%%%%%%

We use the notation and the definitions from 
\cite[Section 2]{CDS2022}.

For an abstract group $G$, its neutral element is denoted by $1_G$, except for when $1_G$ coincides with $1$ of the field of complex numbers $\C$, in which case the index $G$ is omitted.

Let $\T$ denote the \emph{group of complex units}, or the group of elements $\{z\in\C : |z|=1\}$. The group $\T_n$ is a group of $n$-th roots of $1\in\C$, so $\T_n=\{z\in C: z^n=1\}$. For example, $\T_2=\{-1,1\}$, or $\T_4=\{i,-1,-i,1\}$.

$M_{n,m}(F)$ denotes a set of matrices of size $n\times m$ with entries in $F$, typically $F\in\{\C, \C G\}$. We assume $M_n(F):=M_{n,n}(F)$. Throughout the paper, the identity matrix of size $n$ is denoted by $I_{M_n(F)}$, while $J_{M_n(F)}$ is a matrix in $M_n(F)$ such that all its entries are $1_G$ in case $F=\C G$ or $1$ in case $F=\C$.

A \emph{gain graph} over a group $G$, also referred to as \emph{$G$-gain graph}, is a pair $(\Gamma,\psi)$, where $\Gamma=(V,A)$ is an underlying directed graph such that any arc $(u,w)\in A$ has a reverse $(w,u)\in A$, and $\psi:A\to G$ is a gain function such that for any arc $(u,w)\in A$ we have $\psi(w,u)=\psi(u,v)^{-1}$. In particular, a $\T_2$-gain graph is a graph with each arc labelled by $1$ or $-1$ (or, alternatively, a sign $+$ or $-$), which is also called a \emph{sign graph}. The concept of gain graphs has been first introduced by Zaslavsky~\cite{z1989}. For a regularly updated bibliography on signed and gain graphs we refer to~\cite{zaslavskybib}.

For a $G$-gain graph $(\Gamma,\psi)$ on $n$ vertices $V=\{v_1,v_2,\dots,v_n\}$ we define its \emph{adjacency matrix} $A_{(\Gamma,\psi)}$ with entries
$$\left(A_{(\Gamma,\psi)}\right)_{i,j}=\begin{cases}
\psi(v_i,v_j), & v_i\sim v_j; \\
0, & \text{otherwise.}
\end{cases}$$
This way, $A_{(\Gamma,\psi)}\in M_n(\C G)$. It follows from the definition that $A^*_{(\Gamma,\psi)}=A_{(\Gamma,\psi)}$.

Similar to how ordinary graphs are usually considered up to isomorphism, gain graphs are typically considered up to switching isomorphism \cite{z1981}. Two gain functions $\psi_1$ and $\psi_2$ on the same underlying graph $\Gamma=(V,A)$ are \emph{switching equivalent} if there exists a function $f:V\to G$ such that for any pair of adjacent vertices $v,w$ we have
$$\psi_2(v,w)=f(v)^{-1}\psi_1(v,w)f(w).$$
Two gain graphs $(\Gamma_1,\psi_1)$ and $(\Gamma_2,\psi_2)$ with vertex sets $V_1$ and $V_2$ respectively are then \emph{switching isomorphic} if there exists a graph isomorphism $\phi:V_1\to V_2$ and the gain function $\psi_1$ is switching equivalent to $\psi_2\circ\phi$ defined by $(\psi_2\circ\phi)(v,w)=\psi_2(\phi(v),\phi(w))$ for any arc $(v,w)$ in $\Gamma_1$.

It is known that the switching equivalence relation can be expressed in terms of adjacency matrices of the gain graphs on $n$ vertices. Namely, $\psi_1$ and $\psi_2$ are switching equivalent if and only if there exists a diagonal matrix $F\in M_n(\C G)$ such that $A_{(\Gamma,\psi_2)}=F^*A_{(\Gamma,\psi_1)}F$ \cite[Theorem 4.1]{CDD2021}. It follows that the gain graphs $(\Gamma_1,\psi_1)$ and $(\Gamma_2,\psi_2)$ are switching isomorphic if and only if there exists a diagonal matrix $F\in M_n(\C G)$ with $F_{i,i}\in G$ for $i\in\{1,\dots,n\}$ and a permutation matrix $P\in M_n(\C G)$ with entries $0$ and $1_G$ such that (see \cite[Remark 2.4]{CDS2022}) $$A_{(\Gamma_2,\psi_2)}=(PF)^*A_{(\Gamma_1,\psi_1)}(PF).$$

In case when $G\subseteq\C$, it is natural to define two $G$-gain graphs to be cospectral if the multisets of eigenvalues of their adjacency matrices (both in $M_n(\C)$) coincide since the characteristic polynomial and the eigenvalues of a complex-valued matrix are well understood. This is the definition used in~\cite{Belardo2,R2012} where $\T_n$-gain graphs were considered. In a general case when $G$ is not necessarily a group of complex numbers, we define $G$-cospectrality through the equality of the sequences of trace powers (the spectral moments) of the two adjacency matrices up to conjugacy.

To be more precise, let $Tr: M_n(\C G)\to \C G$ be a function that maps matrix $A$ to $\sum_{i=1}^n A_{i,i}$. Also, let $[g]$ be a conjugacy class of $g\in G$ and let $\C_{Class}[G]$ be a set of finitely supported \emph{class functions}, \emph{i.e.} functions $f:G\to\C$ such that $f(g_1)=f(g_2)$ whenever $g_1,g_2$ are in the same conjugacy class $[g]$. We consider a natural map $\mu:\C G\to \C_{Class}[G]$ defined by $$\mu\left(\sum\limits_{x\in G} a_x x\right)(g)=\sum\limits_{x\in[g]}a_x.$$

Then two matrices $A,B\in M_n(\C G)$ are \emph{$G$-cospectral} if $\mu(Tr(A^h))=\mu(Tr(B^h))$ for any positive integer $h$, and two $G$-gain graphs $(\Gamma_1,\psi_1)$ and $(\Gamma_2,\psi_2)$ are \emph{$G$-cospectral} if their adjacency matrices are $G$-cospectral.

Note that, similarly to ordinary graphs, for a gain graph $(\Gamma,\psi)$ with vertex set $V=\{v_1,v_2,\dots,v_n\}$ and a positive integer $h$ the entry $(A_{(\Gamma,\psi)}^h)_{i,j}$ corresponds to the sum of gains of all walks of length $h$ from $v_i$ to $v_j$~\cite[Lemma 4.1]{CDD2021}. As such, the definition of $G$-cospectrality is a generalization of the cospectrality notion of gain graphs in the case of complex-valued adjacency matrices.

An instrumental tool in asserting $G$-cospectrality of gain graphs constructed in Section~\ref{sec:Gcosp} is the following lemma.
\begin{lemma}\cite[Lemma 2.6]{CDS2022} Let $(\Gamma_1,\psi_1)$ and $(\Gamma_2,\psi_2)$ be two $G$-gain graphs with adjacency matrices $A_1,A_2\in M_n(\C G)$, respectively. Let $Q,R\in M_n(\C G)$ be such that every entry of $R$ is a complex multiple of $1_G$ and $QR=RQ=I_{M_n(\C G)}$.
If $A_2=QA_1R$ then $(\Gamma_1,\psi_1)$ and $(\Gamma_2,\psi_2)$ are $G$-cospectral.
\end{lemma}

A \emph{representation} of a group $G$ of degree $k$ is a group homomorphism $\pi:G\to GL_k(\C)$, where $GL_k(\C)$ is a set of all invertible matrices in $M_k(\C)$. An easy example is for a symmetric group on $k$ elements $S_k$. A homomorphism $\pi:S_k\to GL_k(\C)$ that maps a permutation in $S_k$ to a respective permutation matrix of size $k$ is a representation. Moreover, it is a \emph{unitary representation}, \emph{i.e.} a representation such that $\pi(g)\in U_k(\C)$ for any $g\in G$, where $U_k(\C)=\{M\in GL_k(\C): M^{-1}=M^*\}$ is a set of unitary matrices of size $k$ over $\C$. It is a well-known fact that any finite group $G$ can be isomorphically embedded into a symmetric group $S_|G|$, so a unitary representation always exists for any finite group. Another example is a \emph{trivial representation} $\pi_0:G\to C$ that maps any $g\in G$ to $1$ which is also unitary. In case $G\subseteq \C$ one can also consider an \emph{identical representation} $\pi_{id}:G\to GL_1(C)$ such that $\pi_{id}(g)=g$ for any $g\in G$.

Along with a group homomorphism $\pi: G\to GL_k(\C)$, $\pi$ will also sometimes denote its natural linear extension to $\C G$, which is an algebra homomorphism. Namely,
\begin{align*}
    \pi: \C G &\to M_k(\C) \text{ s.t.} \\
    \sum\limits_{g\in G} a_g g &\mapsto \sum\limits_{g\in G} a_g \pi(g).
\end{align*}
This can be even further extended to $\pi:M_{n,m}(\C G)\to M_{nk,mk}(\C)$, where $A\in M_{n,m}(\C G)$ is mapped to a matrix $\pi(A)$ obtained by replacing each entry $A_{i,j}\in\C G$ with a matrix $\pi(A_{i,j})\in M_k(\C)$.

Let $(\Gamma,\psi)$ be a $G$-gain graph, and let $\pi$ be a unitary representation of $G$. We say $\pi(A_{(\Gamma,\psi)})$ is the \emph{represented adjacency matrix} of $(\Gamma,\psi)$ with respect to $\pi$. According to \cite[Proposition 3.4]{CDD2021}, $\pi(A_{(\Gamma,\psi)})$ is Hermitian since $\pi$ is unitary, and its (real) spectrum is called the \emph{$\pi$-spectrum} of $(\Gamma,\psi)$. Finally, two $G$-gain graphs $(\Gamma_1,\psi_1)$ and $(\Gamma_2,\psi_2)$ are \emph{$\pi$-cospectral} if they have the same $\pi$-spectrum.

It follows from \cite[Theorem 4.14]{CD2022} that, for a finite group $G$, two $G$-gain graphs $(\Gamma_1,\psi_1)$ and $(\Gamma_2,\psi_2)$ are $G$-cospectral if any only if they are $\pi$-cospectral for every unitary representation $\pi$ of $G$. This way, $G$-cospectrality is a more general notion than $\pi$-cospectrality. % Theorem 4.14 is given for finite G.

%%%%%%%%%%%%%%%%%%%%%%%%%%%%%%%%%%%%%%%%%%%%%%%%%
%\subsection{Existing methods}
%%%%%%%%%%%%%%%%%%%%%%%%%%%%%%%%%%%%%%%%%%%%%%%%%

%\textcolor{red}{can we merge this in the introduction without introducing many technical definitions (just the basic)?} \textcolor{magenta}{moved to introduction}
 
%Adaptations of WQH-switching for gain graphs have been described in the literature for some particular cases. In \cite[Section 4]{Belardo1}, WQH-switching (referred to as modified Godsil-McKay switching) was generalized to signed graphs, which can be regarded as $\T_2$-gain graphs. In \cite[Section 3]{Belardo2}, the same switching was described for $\T_n$-gain graphs. In both cases, adaptations of WQH-switching were used to construct $\pi_{id}$-cospectral graphs, where $\pi_{id}$ is the identical representation. For a general group $G\not\subseteq\C$, an approach based on group representations was introduced in \cite{CDD2021} in order to discuss the spectrum of a $G$-gain graph with respect to a representation $\pi$. In \cite{CD2022}, the notion of $G$-cospectrality independent of the choice of $\pi$ was first introduced, and in \cite{CDS2022}, generalizations of GM-switching with respect to both $G$- and $\pi$-cospectrality were described.

%%%%%%%%%%%%%%%%%%%%%%%%%%%%%%%%%%%%%%%%%%%%%%%%
\section{A switching to construct $G$-cospectral gain graphs}\label{sec:Gcosp}
%%%%%%%%%%%%%%%%%%%%%%%%%%%%%%%%%%%%%%%%%%%%%%%%%

The main goal of this section is to show a new method to obtain pairs of $G$-cospectral gain graphs. We are inspired by the
Godsil-McKay switching \cite{gm1982}, the Wang-Qiu-Hu switching \cite{WQH2019,QJW2020} and its generalizations \cite{Belardo1,Belardo2}.

Let $(\Gamma,\psi)$ be a $G$-gain graph on $n$ vertices, and suppose that $$\alpha=\{C_0,C_1,C_2,\dots,C_{2k-1},C_{2k}\}$$ is a partition of the vertex set of $V_\Gamma$. With respect to $\alpha$ for every vertex $v$, we define
$$\Psi_i(v):=\sum\limits_{w\in C_i, w\sim v} \psi(v,w).$$

\begin{definition}\label{d:GWQH}
A partition $\alpha$ is a \emph{$G$-WQH partition} if the following conditions hold:
\begin{itemize}
    \item $|C_0|=n_0$ and $|C_i|=|C_{i+1}|=n_i$ for any odd $i<2k$;
    \item for $i,j\in\{1,\dots,2k\}$ and $v,v'\in C_i$ we have $\Psi_j(v)=\Psi_j(v')$; %within each block the row/col sum is constant.
    \item for odd $i,j<2k$ and $v\in C_i$, $v'\in C_{i+1}$ we have $\Psi_j(v)=\Psi_{j+1}(v')$ and \linebreak $\Psi_{j+1}(v)=\Psi_{j}(v')$; 
    \item for every $v\in C_0$ and an odd $i<2k$ we have either
    \begin{enumerate}
        \item[(a)] $\Psi_i(v)=\Psi_{i+1}(v)$, or
        \item[(b)] $\Psi_i(v)=|C_i|g_1$ and $\Psi_{i+1}(v)=|C_{i+1}|g_2$ for some distinct $g_1,g_2\in G\cup\{0\} \subset \C G$.
    \end{enumerate} %the neighbors condition
\end{itemize}
\end{definition}

Alternatively, in a $G$-WQH partition, the total gain over all edges from $v\in C_i$ to vertices in $C_j$ does not depend on the choice of $v$. Additionally, for odd $i,j<2k$ the pairs of subsets $C_i\cup C_{i+1}$ and $C_j\cup C_{j+1}$ are subject to the following relation: the total gain over all edges from $v\in C_i$ to vertices in $C_j$ ($C_{j+1}$) must be the same as the total gain over all edges from $v'\in C_{i+1}$ to vertices in $C_{j+1}$ ($C_j$). 
Moreover, for a vertex $v\in C_0$ and an odd $i<2k$, either the total gain summed over all edges to $C_i$ is the same as the total gain summed over all edges to $C_{i+1}$, or such $v$ must be adjacent to all vertices of $C_i$ with the same gain $g_1$ and to all vertices of $C_{i+1}$ with the same gain $g_2$, where either gain may be zero, implying non-adjacency. The requirement for $g_1$ and $g_2$ to be distinct is stated to prevent overlap between the two cases, although throughout the following definitions and proofs no conflict arises if $g_1=g_2$.

\begin{example}\label{ex:GWQH}
Let $(\Gamma,\psi)$ be a $\T_4$-gain graph on $13$ vertices $v_0,v_1,\dots,v_{12}$ with adjacency matrix:
$$A_{(\Gamma,\psi)}=\left(\begin{tabular}{c||ccc|ccc||ccc|ccc}
$0$   & $1$ & $1$ & $1$     & $0$ & $0$ & $0$     & $1$ & $0$ & $0$     & $1$ & $0$ & $0$ \\ \hline\hline
$1$   & $0$ & $0$ & $0$     & $0$ & $0$ & $0$     & $0$ & $0$ & $0$     & $0$ & $0$ & $0$ \\
$1$   & $0$ & $0$ & $0$     & $0$ & $0$ & $0$     & $0$ & $0$ & $0$     & $0$ & $0$ & $0$ \\
$1$   & $0$ & $0$ & $0$     & $0$ & $0$ & $0$     & $0$ & $0$ & $0$     & $0$ & $0$ & $0$ \\ \hline
$0$   & $0$ & $0$ & $0$     & $0$ & $i$ & $-i$    & $0$ & $0$ & $0$     & $0$ & $0$ & $0$ \\
$0$   & $0$ & $0$ & $0$     & $-i$ & $0$ & $i$    & $0$ & $0$ & $0$     & $0$ & $0$ & $0$ \\
$0$   & $0$ & $0$ & $0$     & $i$ & $-i$ & $0$    & $0$ & $0$ & $0$     & $0$ & $0$ & $0$ \\ \hline\hline
$1$   & $0$ & $0$ & $0$     & $0$ & $0$ & $0$     & $0$ & $0$ & $0$     & $0$ & $0$ & $0$ \\
$0$   & $0$ & $0$ & $0$     & $0$ & $0$ & $0$     & $0$ & $0$ & $0$     & $0$ & $0$ & $0$ \\
$0$   & $0$ & $0$ & $0$     & $0$ & $0$ & $0$     & $0$ & $0$ & $0$     & $0$ & $0$ & $0$ \\ \hline
$1$   & $0$ & $0$ & $0$     & $0$ & $0$ & $0$     & $0$ & $0$ & $0$     & $0$ & $i$ & $-i$ \\
$0$   & $0$ & $0$ & $0$     & $0$ & $0$ & $0$     & $0$ & $0$ & $0$     & $-i$ & $0$ & $i$ \\
$0$   & $0$ & $0$ & $0$     & $0$ & $0$ & $0$     & $0$ & $0$ & $0$     & $i$ & $-i$ & $0$ \\
\end{tabular}\right).$$
Then the partition $\alpha$ into subsets
\begin{gather*}
C_0=\{v_0\},\; C_1=\{v_1,v_2,v_3\},\;C_2=\{v_4,v_5,v_6\},\\
C_3=\{v_7,v_8,v_9\},\;C_4=\{v_{10},v_{11},v_{12}\}    
\end{gather*}
is a $G$-WQH partition. Indeed, we have $|C_1|=|C_2|$ and $|C_3|=|C_4|$. Moreover, it is easily verified that for any $i,j\in\{1,2,3,4\}$ and any vertex $w_i\in C_i$ we have $\Psi_j(w_i)=0$, so the second and the third conditions in Definition~\ref{d:GWQH} hold up. Finally, for $v_0\in C_0$ and the subsets $C_3$ and $C_4$ we have $\Psi_3(v_0)=\Psi_4(v_0)=1$, which satisfies $(a)$ of the fourth condition, and for the subsets $C_1$ and $C_2$ we have $\Psi_1(v_0)=3=3\cdot1$ and $\Psi_2(v_0)=0=3\cdot 0$, which satisfies $(b)$ of that condition for $g_1=1\in\T_4$ and $g_2=0$.
\end{example}

\begin{definition}\label{d:alphagraph}
For a $G$-WQH partition $\alpha$ and a $G$-gain graph $(\Gamma,\psi)$, we define a \emph{gain graph} $(\Gamma^\alpha,\psi^\alpha)$ as follows (assuming $\psi(v,w)=0$ iff $v\not\sim w$):
%\textcolor{red}{is this not implicit in the definition of adjacency matrix of $(\Gamma,\psi)$?} \textcolor{magenta}{It isn't, psi is a function to G, not G\cup0}):
\begin{itemize}
    \item for $v,w\in C_1\cup\cdots\cup C_{2k}$, we have $\psi^\alpha(v,w)=\psi(v,w)$ (that is, the adjacency and the gains between pairs of vertices in $\bigcup\limits_{i=1}^{2k}C_i$ are the same as in $(\Gamma,\psi)$);
    \item for $v\in C_0$ and an odd $i<2k$ such that $\Psi_i(v)=\Psi_{i+1}(v)$ and $w\in C_i\cup C_{i+1}$, we have $\psi^\alpha(v,w)=\psi(v,w)$; 
    %\textcolor{red}{do we need this condition?} \textcolor{magenta}{I think so? To cover all pairs of vertices}
    \item for $v\in C_0$ and an odd $i<2k$ such that $\Psi_i(v)=|C_i|g_1$ and $\Psi_{i+1}(v)=|C_{i+1}|g_2$ for some $g_1,g_2\in G$ we have $\psi^\alpha(v,w)=g_2$ if $w\in C_i$ and $\psi^\alpha(v,w)=g_1$ if $w\in C_{i+1}$.
\end{itemize}
\end{definition}

The $G$-WQH partition is said to be \emph{nontrivial} if $(\Gamma,\psi)$ and $(\Gamma^\alpha,\psi^\alpha)$ are not switching isomorphic.

Finally, for a $G$-WQH partition $\alpha$, let $Q_\alpha \in M_n(\C G)$ be a block-diagonal matrix $Q_\alpha =\operatorname{diag}\left(I_{M_{n_0}(\C G)}, Q_{n_1},\dots,Q_{n_k}\right)$, where for each $i=1,\dots,k$
$$Q_{n_i}=\left(\begin{matrix}
I_{M_{n_i}(\C G)}-\frac1{n_i}J_{M_{n_i}(\C G)} & \frac1{n_i}J_{M_{n_i}(\C G)} \\
\frac1{n_i}J_{M_{n_i}(\C G)} & I_{M_{n_i}(\C G)}-\frac1{n_i}J_{M_{n_i}(\C G)}
\end{matrix}\right),$$
where recall that $n_i$ is the size of $C_i$ and $C_{i+1}$ for any odd $i<2k$.

Observe that the nonzero entries of $Q_\alpha$ are all real multiples of $1_G$ and $Q_\alpha^*=Q_\alpha$.

\begin{lemma}\label{l:block} Let $A$ be a $2n_i\times 2n_j$ block matrix $$A=\left(\begin{matrix}
A_{1,1} & A_{1,2} \\
A_{2,1} & A_{2,2}
\end{matrix}\right)$$
with each block has size $n_i\times n_j$ and such that the blocks $A_{1,1}$ and $A_{2,2}$ (as well as $A_{1,2}$ and $A_{2,1}$) have constant row sum of the same value and constant column sum of the same value.
Then $Q_{n_i} A Q_{n_j} = A$.
\end{lemma}

\begin{proof}
Suppose $r_1$ and $c_1$ are, respectively, row sum and column sum constants for $A_{1,1}$ and $A_{2,2}$, while $r_2$ and $c_2$ are such constants for $A_{1,2}$ and $A_{2,1}$. Observe that
{\small{
\begin{align*}
  &J_{M_{n_i}(\C G)} A_{l,l} = c_1 J_{M_{n_i,n_j}(\C G)},\quad
  A_{l,l} J_{M_{n_j}(\C G)}= r_1 J_{M_{n_i,n_j}(\C G)} \text{ for }l\in\{1,2\}; \\
  &J_{M_{n_i}(\C G)} A_{l,m} = c_2 J_{M_{n_i,n_j}(\C G)},\quad
  A_{l,m} J_{M_{n_j}(\C G)}= r_2 J_{M_{n_i,n_j}(\C G)} \text{ for }\{l,m\}=\{1,2\}; \\
  &J_{M_{n_i}(\C G)} J_{M_{n_i,n_j}(\C G)} = n_i J_{M_{n_i,n_j}(\C G)},\quad
  J_{M_{n_i,n_j}(\C G)} J_{M_{n_j}(\C G)}= n_j J_{M_{n_i,n_j}(\C G)}. 
\end{align*}
}}
Using the above equalities, it is straightforward to verify that $Q_{n_i} A Q_{n_j} = A$ by performing the block matrix multiplication.
\end{proof}

\begin{theorem}\label{t:gcosp} Let $(\Gamma,\psi)$ be a $G$-gain graph and let $\alpha$ be a $G$-WQH partition with associated matrix $Q_\alpha\in M_n(\C G)$. Then
$$A_{(\Gamma^\alpha,\psi^\alpha)}=Q_\alpha A_{(\Gamma,\psi)} Q_\alpha.$$
In particular, $(\Gamma,\psi)$ and $(\Gamma^\alpha,\psi^\alpha)$ are $G$-cospectral.
\end{theorem}

\begin{proof} With a suitable labeling of vertices of $\Gamma$ we can write
$$A_{(\Gamma,\psi)}=\left(\begin{matrix}
C_{0,0} & C_{0,1} & \dots & C_{0,k} \\
C_{0,1}^* & C_{1,1} & \dots & C_{1,k} \\
\vdots & \cdots & \ddots & \vdots \\
C_{0,k}^* & C_{1,k}^* & \cdots & C_{k,k}
\end{matrix}\right).$$
Here, for all $i,j\in\{1,\dots,k\}$ the block $C_{i,j}\in M_{2n_{2i-1},2n_{2j-1}}(\C G)$ describes the adjacencies and gains from vertices in $C_{2i-1}\cup C_{2i}$ to vertices in $C_{2j-1}\cup C_{2j}$. The block $C_{0,j}\in M_{n_0,2n_{2j-1}}(\C G)$ describes the adjacencies and gains from vertices in $C_0$ to vertices in $C_{2j-1}\cup C_{2j}$. Later we will also use a split of this block $C_{0,j}=\left( C_{0,j}^{(1)}, \; C_{0,j}^{(2)} \right)$, where $C_{0,j}^{(1)}$ describes adjacencies between the vertices of $C_0$ and $C_{2j-1}$, while $C_{0,j}^{(2)}$ does so for $C_0$ and $C_{2j}$. Finally, the block $C_{0,0}\in M_{n_0}(\C G)$ describes adjacencies and gains within the vertex subset $C_0$.

The following is true for $A_{(\Gamma,\psi)}$ by the definition of $G$-WQH partition (Definition \ref{d:GWQH}):
\begin{itemize}
    \item for $i\in\{0,\dots, k\}$ we have $C_{i,i}^*=C_{i,i}$;
    \item for $i,j\in\{0,\dots,k\}$ the matrix $C_{i,j}$ satisfies the conditions of Lemma~\ref{l:block}.
\end{itemize}

By block matrix multiplication and from Lemma~\ref{l:block} we have: %(assuming $J:=J_{M_{n_1}(\C G)}$ and using $J^2=n_1 J$ as well as $BJ=JB$ for any $B\in\{C_{1,1}, C_{1,2}, C_{2,2}\}$):
%{\small
%$$Q_{\alpha} A_{(\Gamma,\psi)} Q_{\alpha}=\left(\begin{matrix}
%C_{0,0} & C_{0,1}-\frac1{n_1} C_{0,1} J + \frac1{n_1} C_{0,2}J & \frac1{n_1} C_{0,1} J + C_{0,2} - \frac1{n_1} C_{0,2}J \\
%C_{0,1}^* -\frac1{n_1} J C_{0,1}^* + \frac1{n_1} J C_{0,2}^* & C_{1,1} & C_{1,2}\\
%\frac1{n_1} J C_{0,1}^* + C_{0,2}^* - \frac1{n_1}J C_{0,2}^* & C_{1,2}^* & C_{2,2}
%\end{matrix}\right),$$
%}
$$Q_{\alpha} A_{(\Gamma,\psi)} Q_{\alpha}=\left(\begin{matrix}
C_{0,0} & C_{0,1}Q_{n_1} & \dots & C_{0,k}Q_{n_k} \\
Q_{n_1}C_{0,1}^* & C_{1,1} & \dots & C_{1,k} \\
\vdots & \cdots & \ddots & \vdots \\
Q_{n_k}C_{0,k}^* & C_{1,k}^* & \cdots & C_{k,k}
\end{matrix}\right).$$

Let $v\in C_0$ and $j\in\{1,\dots,k\}$, and let ${\bf x}$ be the row of $C_{0,j}$ which corresponds to $v$ (the following observations are just as valid for the column ${\bf x}^*$ of $C_{0,j}^*$ corresponding to $v$).
With the use of the split $C_{0,j}=\left( C_{0,j}^{(1)}, \; C_{0,j}^{(2)} \right)$ we can write down (assuming $J:=J_{M_{n_j}(\C G)}$)
$$C_{0,j}=\left( C_{0,j}^{(1)}-\frac1{n_j} C_{0,j}^{(1)} J + \frac1{n_j} C_{0,j}^{(2)} J, \quad \frac1{n_j} C_{0,j}^{(1)} J + C_{0,j}^{(2)} - \frac1{n_j} C_{0,j}^{(2)} J \right).$$

There are two possible cases.
First, if $\Psi_{2j-1}(v)=\Psi_{2j}(v)$, then the row corresponding to $v$ in $A_{(\Gamma^\alpha,\psi^\alpha)}$ is also ${\bf x}$ as the contribution of $\frac1{n_j} C_{0,j}^{(1)}J$ and $\frac1{n_j} C_{0,j}^{(2)}J$ terms amounts to zero in this case.
On the other hand, suppose $\Psi_{2j-1}(v)=n_{2j-1} g_1$ and $\Psi_{2j}(v)=n_{2j-1} g_2$ for some $g_1,g_2\in G\cup\{0\}$, or, in other words, ${\bf x}$ takes form $(g_1,\dots,g_1,g_2,\dots,g_2)$, a row of $n_{2j-1}$ entries $g_1$ followed by $n_{2j-1}$ entries $g_2$. Then the row of $C_{0,j}Q_{n_j}$ corresponding to the vertex $v$ takes form $(g_2,\dots,g_2,g_1,\dots,g_1)$. Either case is consistent with the described construction of the graph $(\Gamma^\alpha, \psi^\alpha)$, and so the matrix $Q_{\alpha} A_{(\Gamma,\psi)} Q_{\alpha}$ is indeed equal to the adjacency matrix $A_{(\Gamma^\alpha,\psi^\alpha)}$.
\end{proof}

\begin{example}
Let $(\Gamma,\psi)$ be a $\T_4$-gain graph on $13$ vertices with the partition $\alpha$ as in Example~\ref{ex:GWQH}. It was already shown that $\alpha$ is a $G$-WQH partition. Moreover, the partition is nontrivial. Indeed, the graph $(\Gamma^\alpha,\psi^\alpha)$ is obtained from $(\Gamma,\psi)$ by removing edges $\{v_0,v_i\}$, $i\in\{1,2,3\}$ and adding new edges $\{v_0,v_i\}$, $i\in\{4,5,6\}$ all with gain $1$. This graph is disconnected unlike $(\Gamma,\alpha)$, and hence cannot be switching isomorphic to it. It also cannot be constructed using the GM switching for $G$-cospectral graphs from~\cite{CDS2022} since there is no suitable $G$-GM partition of $V_\Gamma$.
\end{example}

%%%%%%%%%%%%%%%%%%%%%%%%%%%%%%%%%%%%%%%%%%%%%%%%%
\section{A switching with respect to a representation: \linebreak $\pi$-cospectral gain graphs}
%%%%%%%%%%%%%%%%%%%%%%%%%%%%%%%%%%%%%%%%%%%%%%%%%

In this section, we present a method to obtain pairs of $\pi$-cospectral gain graphs for some unitary representation $\pi$ of the group $G$.

Let $(\Gamma,\psi)$ be a $G$-gain graph on $n$ vertices, and let $\pi$ be a unitary representation $\pi:G\to U_k(\C)$ ($\pi$ also stands for a linear extension $\pi:\C G\to M_k(\C)$). Recall the function
$$\Psi_i(v):=\sum\limits_{w\in C_i, w\sim v} \psi(v,w)$$
defined with respect to a partition $\alpha=\{C_0, C_1,\dots,C_{2k}\}$. In this section, we will also use a function $\Psi_{i,i+1}(v)=\Psi_i(v)+\Psi_{i+1}(v)$ for an odd $i<2k$.

\begin{remark}
If $\pi_0$ is the trivial representation mapping each element of $G$ to $1$, then
$\pi_0(\psi_i(v))$ is the number of vertices of $C_i$ which are adjacent to $v$.
\end{remark}

\begin{definition}\label{d:piWQH}
A partition $\alpha=\{C_0,C_1,\dots,C_{2k}\}$ of the vertex set of $\Gamma$ is a \emph{$\pi$-WQH partition} if the following conditions hold:
\begin{itemize}
    \item $|C_0|=n_0$ and $|C_i|=|C_{i+1}|=n_i$ for an odd $i<2k$;
    \item for $i,j\in\{1,\dots,2k\}$ and $v,v'\in C_i$ we have $\pi(\Psi_j(v))=\pi(\Psi_j(v'))$;
    \item for odd $i,j<2k$ and $v\in C_i$, $v'\in C_{i+1}$ we have $\pi(\Psi_j(v))=\pi(\Psi_{j+1}(v'))$ and $\pi(\Psi_{j+1}(v))=\pi(\Psi_{j}(v'))$;
    \item for every $v\in C_0$ and an odd $i<2k$ we have either
    \begin{enumerate}
        \item[(a)] $\Psi_i(v)=\Psi_{i+1}(v)$, or
        \item[(b)] $\Psi_i(v)=|C_i|g_1$ and $\Psi_{i+1}(v)=|C_{i+1}|g_2$ for some distinct $g_1,g_2\in G\cup\{0\}$.
    \end{enumerate}
\end{itemize}
\end{definition}

The above definition is closely related to Definition~\ref{d:GWQH}. For a $\pi$-WQH partition $\alpha$, the graph $(\Gamma^\alpha,\psi^\alpha)$ is constructed exactly as described in Definition~\ref{d:alphagraph}.

\begin{example}\label{ex:piWQH}
This example builds on~\cite[Example 4.3]{CDS2022}. Consider an $S_4$-gain graph $(\Gamma,\psi)$ depicted in the left of Figure~\ref{fig:picosp}; the $\psi$-image of any unlabeled edge is $1_{S_4}$. Let $\alpha$ be a partition $\{C_0,C_1,C_2,C_3,C_4\}$, where $C_0=\{v_1\}$, $C_1=\{v_2,v_3,v_4,v_5\}$, $C_2=\{v_6,v_7,v_8,v_9\}$, $C_3=\{v_{10},v_{11},v_{12},v_{13}\}$, and $C_4=\{v_{14},v_{15},v_{16},v_{17}\}$. It is not a $G$-WQH partition, since $\Psi_1(v)=1_{S_4}+(12)(34)\neq (12)+(34)=\Psi_2(v')$ for $v\in C_1$ and $v'\in C_2$. However, by choosing a unitary representation $\pi$ that sends a permutation from $S_4$ to a respective $4\times 4$ permutation matrix, we obtain
\begin{align*}
\pi(1_{S_4})&=\left(\begin{matrix}
1 & 0 & 0 & 0 \\
0 & 1 & 0 & 0 \\
0 & 0 & 1 & 0 \\
0 & 0 & 0 & 1
\end{matrix}\right),\; \pi((12)(34))=\left(\begin{matrix}
0 & 1 & 0 & 0 \\
1 & 0 & 0 & 0 \\
0 & 0 & 0 & 1 \\
0 & 0 & 1 & 0
\end{matrix}\right), \\
\pi((12))&=\left(\begin{matrix}
0 & 1 & 0 & 0 \\
1 & 0 & 0 & 0 \\
0 & 0 & 1 & 0 \\
0 & 0 & 0 & 1
\end{matrix}\right),\; \pi((34))=\left(\begin{matrix}
1 & 0 & 0 & 0 \\
0 & 1 & 0 & 0 \\
0 & 0 & 0 & 1 \\
0 & 0 & 1 & 0
\end{matrix}\right).
\end{align*}
and so $\pi(\Psi_1(v))=\pi(\Psi_2(v'))$. Similarly, for $v\in C_3$ and $v'\in C_4$ we have $\pi(\Psi_3(v))=\pi(\Psi_4(v'))$. As for the other conditions, we have $|C_1|=|C_2|$ and $|C_3|=|C_4|$, as well as $\Psi_j(v_i)=0$ for any distinct $i,j\in\{1,2,3,4\}$, so the first three conditions in Definition~\ref{d:piWQH} are satisfied. For the final condition, we have $\Psi_1(v_1)=4\times 1_{S_4}$ and $\Psi_2(v_1)=4\times 0$ which satisfies the $(a)$ part, and $\Psi_3(v_1)=\Psi_4(v_1)=1_{S_4}$ satisfying part $(b)$.
\end{example}

\begin{theorem}\label{t:picosp} Let $(\Gamma,\psi)$ be a $G$-gain graph, let $\pi$ be a unitary representation of $G$, and let $\alpha$ be a $\pi$-WQH partition. Then
$$\pi(A_{(\Gamma^\alpha,\psi^\alpha)})=\pi(Q_\alpha)\pi(A_{(\Gamma,\psi)})\pi(Q_\alpha),$$
and in particular, $(\Gamma,\psi)$ and $(\Gamma^\alpha,\psi^\alpha)$ are $\pi$-cospectral.
\end{theorem}

\begin{proof}
Observe that (the extension to $\C G$ of) $\pi$ is a homomorphism,  so \linebreak $\pi(Q_\alpha)\pi(A_{(\Gamma,\psi)})\pi(Q_\alpha)=\pi(Q_\alpha A_{(\Gamma,\psi)} Q_\alpha)$. Also, by the argument similar to one presented in the proof of Theorem~\ref{t:gcosp}, the first $n_0$ rows and columns of matrix $Q_\alpha A_{(\Gamma,\psi)} Q_\alpha$ relate to the adjacencies between vertices from $C_0$ and $C_i$ for $i\in\{0,1,\dots,2k\}$ in $(\Gamma^\alpha, \psi^\alpha)$. Hence all that is left to prove is that the $\pi$-image of $(n-n_0)\times (n-n_0)$ principal submatrix of $Q_\alpha A_{(\Gamma,\psi)} Q_\alpha$ obtained by removing first $n_0$ rows and columns is the same as the $\pi$-image of a similarly constructed submatrix of $A_{(\Gamma^\alpha, \psi^\alpha)}$. We will do so by considering individual blocks of this matrix.

We fix a pair of odd (not necessarily distinct) integers $i,j<2k$ and consider pairs of subsets $C_i\cup C_{i+1}$ and $C_j\cup C_{j+1}$ with vertices $v_1,\dots, v_{2n_i}$ and $w_1,\dots, w_{2n_j}$ respectively. Let $A'$ denote a $2n_i\times 2n_j$ principal submatrix of $A_{(\Gamma,\psi)}$ obtained by taking rows corresponding to vertices $v_1,\dots v_{2n_i}$ and columns corresponding to vertices $w_1,\dots, w_{2n_j}$. The labeling of the vertices is such that an element $A'_{i,j}$ is $\psi(v_i,w_j)$ if $v_i$ and $w_j$ are adjacent or $A'_{i,j}=0$ otherwise. In this case, after the conjugation $Q_\alpha A_{(\Gamma,\psi)} Q_\alpha$ this block takes form $Q_{n_i} A' Q_{n_j}$.

The matrix $A'$ can be represented as a block matrix
$$A=\left(\begin{matrix}
A_{11} & A_{12} \\
A_{21} & A_{22}
\end{matrix}\right).$$
Let $S_1$ be the sum of all elements in block $A_{11}$ (same as the sum of elements in block $A_{22}$), and $S_2$ be the sum of all elements in block $A_{12}$ ($A_{21}$). Let $D$ be a matrix of four $n_i\times n_j$ blocks
$$D=\left(\begin{matrix}
D_{11} & D_{12} \\
D_{21} & D_{22}
\end{matrix}\right),$$
where
{\footnotesize{
\begin{align*}
D_{11}&=\left(\begin{matrix}
\frac{S_1}{n_i}+\frac{S_1}{n_j}-\Psi_{i}(v_1)-\Psi_{j}^*(w_1) & \frac{S_1}{n_j}-\Psi_{j}^*(w_2) & \dots & \frac{S_1}{n_j}-\Psi_{j}^*(w_{n_j}) \\
\frac{S_1}{n_i}-\Psi_{i}(v_2) & 0 & \dots & 0 \\
\vdots & \dots & \dots & \dots \\
\frac{S_1}{n_i}-\Psi_{i}(v_{n_i}) & 0 & \dots & 0
\end{matrix}\right), \\
D_{12}&=\left(\begin{matrix}
\frac{S_2}{n_i}+\frac{S_2}{n_j}-\Psi_{i+1}(v_1)-\Psi_{j}^*(w_{n_j+1}) & \frac{S_2}{n_j}-\Psi_{j}^*(w_{n_j+2}) & \dots & \frac{S_2}{n_j}-\Psi_{j}^*(w_{2n_j}) \\
\frac{S_2}{n_i}-\Psi_{i+1}(v_2) & 0 & \dots & 0 \\
\vdots & \dots & \dots & \dots \\
\frac{S_2}{n_i}-\Psi_{i+1}(v_{n_i}) & 0 & \dots & 0
\end{matrix}\right), 
\\
D_{21}&=\left(\begin{matrix}
\frac{S_2}{n_i}+\frac{S_2}{n_j}-\Psi_{i}(v_{n_i+1})-\Psi_{j+1}^*(w_1) & \frac{S_2}{n_j}-\Psi_{j+1}^*(w_2) & \dots & \frac{S_2}{n_j}-\Psi_{j+1}^*(w_{n_j}) \\
\frac{S_2}{n_i}-\Psi_{i}(v_{n_i+2}) & 0 & \dots & 0 \\
\vdots & \dots & \dots & \dots \\
\frac{S_2}{n_i}-\Psi_{i}(v_{2n_i}) & 0 & \dots & 0
\end{matrix}\right),
\end{align*}
\begin{align*}
D_{22}&=\left(\begin{matrix}
\frac{S_1}{n_i}+\frac{S_1}{n_j}-\Psi_{i+1}(v_{n_i+1})-\Psi_{j+1}^*(w_{n_j+1}) & \frac{S_1}{n_j}-\Psi_{j+1}^*(w_{n_j+2}) & \dots & \frac{S_1}{n_j}-\Psi_{j+1}^*(w_{2n_j}) \\
\frac{S_1}{n_i}-\Psi_{i+1}(v_{n_i+2}) & 0 & \dots & 0 \\
\vdots & \dots & \dots & \dots \\
\frac{S_1}{n_i}-\Psi_{i+1}(v_{2n_i}) & 0 & \dots & 0
\end{matrix}\right). 
\end{align*}
}}
Observe that $A'+D$ has a constant row and column sum of the same value, as well as its two principal submatrices obtained by removing first (last) $n_i$ rows and $n_j$ columns.
Indeed, using the fact that $\Psi_{i,i+1}(v)$ is the sum of the row of $A'$ that corresponds to the vertex $v$ and $\Psi_{j,j+1}^*(w)$ is the sum of the respective column, it is trivial to confirm that the sum of any row or column of $A'+D$ is equal to $\frac{S_1+S_2}{n_i}$ or $\frac{S_1+S_2}{n_j}$ respectively, while the four blocks have row sum $\frac{S_t}{n_i}$ and column sum $\frac{S_t}{n_j}$ where $t\in\{1,2\}$ depending on the block (analogous to the proof of \cite[Theorem 4.4]{CDS2022}). This means we can apply Lemma~\ref{l:block} to show that $Q_{n_i} (A'+D) Q_{n_j} = A'+D$.

In addition, we have $\pi(D)=0$. This follows from the following equalities that hold for any $v\in C_i$, $v'\in C_{i+1}$, $w\in C_j$, and $w'\in C_{j+1}$ due to the definition of $\pi$-WQH partition:
\begin{gather*}
    \pi(S_1)=\pi(n_i \Psi_{i}(v))=\pi(n_j\Psi_j^*(w))=\pi(n_i\Psi_{i+1}(v'))=\pi(n_j\Psi_{j+1}^*(w')), \\
    \pi(S_2)=\pi(n_i \Psi_{i+1}(v))=\pi(n_j\Psi_j^*(w'))=\pi(n_i\Psi_{i}(v'))=\pi(n_j\Psi_{j+1}^*(w)).
\end{gather*}
From this it is easily derived that the $\pi$-image in any element of $D$ is $0$.

By combining $Q_{n_i} (A'+D) Q_{n_j} = A'+D$ and $\pi(D)=0$, we obtain
\begin{align*}
    \pi(Q_{n_i} A' Q_{n_j}) &= \pi(Q_{n_i} (A'+D-D) Q_{n_j})\\
    &=\pi(Q_{n_i} (A'+D) Q_{n_j})-\pi(Q_{n_i})\cdot0\cdot\pi(Q_{n_j}) \\
    &= \pi(A'+D)=\pi(A')+0=\pi(A').
\qedhere\end{align*}
\end{proof}

\begin{figure}[ht]
    \centering    %\input{ex_picosp}
    \begin{tikzpicture}[x=0.75pt,y=0.75pt,yscale=-1,xscale=1]

%Curve Lines [id:da03658879586306907] 
\draw [color={rgb, 255:red, 74; green, 144; blue, 226 }  ,draw opacity=1 ]   (42,70) .. controls (16,146) and (121,172) .. (171,82) ;
%Curve Lines [id:da33917631404169546] 
\draw [color={rgb, 255:red, 208; green, 2; blue, 27 }  ,draw opacity=1 ]   (481.27,82) .. controls (542.27,172) and (615.27,156) .. (606.27,70) ;
%Curve Lines [id:da3447451314033696] 
\draw [color={rgb, 255:red, 208; green, 2; blue, 27 }  ,draw opacity=1 ]   (481.27,82) .. controls (501.27,106) and (537.27,114) .. (567.27,110) ;
%Curve Lines [id:da6307649523869334] 
\draw [color={rgb, 255:red, 208; green, 2; blue, 27 }  ,draw opacity=1 ]   (481.27,82) .. controls (496.27,54) and (530.27,30) .. (567.27,30) ;
%Straight Lines [id:da5905335284125195] 
\draw    (81,30) -- (120,70) ;
\draw [shift={(120,70)}, rotate = 45.73] [color={rgb, 255:red, 0; green, 0; blue, 0 }  ][fill={rgb, 255:red, 0; green, 0; blue, 0 }  ][line width=0.75]      (0, 0) circle [x radius= 3.35, y radius= 3.35]   ;
\draw [shift={(81,30)}, rotate = 45.73] [color={rgb, 255:red, 0; green, 0; blue, 0 }  ][fill={rgb, 255:red, 0; green, 0; blue, 0 }  ][line width=0.75]      (0, 0) circle [x radius= 3.35, y radius= 3.35]   ;
%Straight Lines [id:da6747886057544017] 
\draw    (42,70) -- (81,110) ;
\draw [shift={(81,110)}, rotate = 45.73] [color={rgb, 255:red, 0; green, 0; blue, 0 }  ][fill={rgb, 255:red, 0; green, 0; blue, 0 }  ][line width=0.75]      (0, 0) circle [x radius= 3.35, y radius= 3.35]   ;
\draw [shift={(42,70)}, rotate = 45.73] [color={rgb, 255:red, 0; green, 0; blue, 0 }  ][fill={rgb, 255:red, 0; green, 0; blue, 0 }  ][line width=0.75]      (0, 0) circle [x radius= 3.35, y radius= 3.35]   ;
%Straight Lines [id:da6991603393741839] 
\draw    (257,30) -- (296,70) ;
\draw [shift={(296,70)}, rotate = 45.73] [color={rgb, 255:red, 0; green, 0; blue, 0 }  ][fill={rgb, 255:red, 0; green, 0; blue, 0 }  ][line width=0.75]      (0, 0) circle [x radius= 3.35, y radius= 3.35]   ;
\draw [shift={(257,30)}, rotate = 45.73] [color={rgb, 255:red, 0; green, 0; blue, 0 }  ][fill={rgb, 255:red, 0; green, 0; blue, 0 }  ][line width=0.75]      (0, 0) circle [x radius= 3.35, y radius= 3.35]   ;
%Straight Lines [id:da2291633858258797] 
\draw    (218,70) -- (257,110) ;
\draw [shift={(257,110)}, rotate = 45.73] [color={rgb, 255:red, 0; green, 0; blue, 0 }  ][fill={rgb, 255:red, 0; green, 0; blue, 0 }  ][line width=0.75]      (0, 0) circle [x radius= 3.35, y radius= 3.35]   ;
\draw [shift={(218,70)}, rotate = 45.73] [color={rgb, 255:red, 0; green, 0; blue, 0 }  ][fill={rgb, 255:red, 0; green, 0; blue, 0 }  ][line width=0.75]      (0, 0) circle [x radius= 3.35, y radius= 3.35]   ;
%Straight Lines [id:da19910701861720348] 
\draw    (257,110) -- (296,70) ;
\draw [shift={(296,70)}, rotate = 314.27] [color={rgb, 255:red, 0; green, 0; blue, 0 }  ][fill={rgb, 255:red, 0; green, 0; blue, 0 }  ][line width=0.75]      (0, 0) circle [x radius= 3.35, y radius= 3.35]   ;
\draw [shift={(257,110)}, rotate = 314.27] [color={rgb, 255:red, 0; green, 0; blue, 0 }  ][fill={rgb, 255:red, 0; green, 0; blue, 0 }  ][line width=0.75]      (0, 0) circle [x radius= 3.35, y radius= 3.35]   ;
%Straight Lines [id:da6648505296601148] 
\draw [color={rgb, 255:red, 74; green, 144; blue, 226 }  ,draw opacity=1 ]   (120,70) -- (171,82) ;
\draw [shift={(171,82)}, rotate = 13.24] [color={rgb, 255:red, 74; green, 144; blue, 226 }  ,draw opacity=1 ][fill={rgb, 255:red, 74; green, 144; blue, 226 }  ,fill opacity=1 ][line width=0.75]      (0, 0) circle [x radius= 3.35, y radius= 3.35]   ;
\draw [shift={(120,70)}, rotate = 13.24] [color={rgb, 255:red, 74; green, 144; blue, 226 }  ,draw opacity=1 ][fill={rgb, 255:red, 74; green, 144; blue, 226 }  ,fill opacity=1 ][line width=0.75]      (0, 0) circle [x radius= 3.35, y radius= 3.35]   ;
%Straight Lines [id:da425414854145983] 
\draw [color={rgb, 255:red, 208; green, 2; blue, 27 }  ,draw opacity=1 ]   (481.27,82) -- (528.27,70) ;
\draw [shift={(528.27,70)}, rotate = 345.68] [color={rgb, 255:red, 208; green, 2; blue, 27 }  ,draw opacity=1 ][fill={rgb, 255:red, 208; green, 2; blue, 27 }  ,fill opacity=1 ][line width=0.75]      (0, 0) circle [x radius= 3.35, y radius= 3.35]   ;
\draw [shift={(481.27,82)}, rotate = 345.68] [color={rgb, 255:red, 208; green, 2; blue, 27 }  ,draw opacity=1 ][fill={rgb, 255:red, 208; green, 2; blue, 27 }  ,fill opacity=1 ][line width=0.75]      (0, 0) circle [x radius= 3.35, y radius= 3.35]   ;
%Straight Lines [id:da8654795638735502] 
\draw [color={rgb, 255:red, 74; green, 144; blue, 226 }  ,draw opacity=1 ]   (81,30) -- (171,82) ;
\draw [shift={(171,82)}, rotate = 30.02] [color={rgb, 255:red, 74; green, 144; blue, 226 }  ,draw opacity=1 ][fill={rgb, 255:red, 74; green, 144; blue, 226 }  ,fill opacity=1 ][line width=0.75]      (0, 0) circle [x radius= 3.35, y radius= 3.35]   ;
\draw [shift={(81,30)}, rotate = 30.02] [color={rgb, 255:red, 74; green, 144; blue, 226 }  ,draw opacity=1 ][fill={rgb, 255:red, 74; green, 144; blue, 226 }  ,fill opacity=1 ][line width=0.75]      (0, 0) circle [x radius= 3.35, y radius= 3.35]   ;
%Straight Lines [id:da6429598250988227] 
\draw [color={rgb, 255:red, 74; green, 144; blue, 226 }  ,draw opacity=1 ]   (81,110) -- (171,82) ;
\draw [shift={(171,82)}, rotate = 342.72] [color={rgb, 255:red, 74; green, 144; blue, 226 }  ,draw opacity=1 ][fill={rgb, 255:red, 74; green, 144; blue, 226 }  ,fill opacity=1 ][line width=0.75]      (0, 0) circle [x radius= 3.35, y radius= 3.35]   ;
\draw [shift={(81,110)}, rotate = 342.72] [color={rgb, 255:red, 74; green, 144; blue, 226 }  ,draw opacity=1 ][fill={rgb, 255:red, 74; green, 144; blue, 226 }  ,fill opacity=1 ][line width=0.75]      (0, 0) circle [x radius= 3.35, y radius= 3.35]   ;
%Straight Lines [id:da17972173890445187] 
\draw    (240,159) -- (279,199) ;
\draw [shift={(279,199)}, rotate = 45.73] [color={rgb, 255:red, 0; green, 0; blue, 0 }  ][fill={rgb, 255:red, 0; green, 0; blue, 0 }  ][line width=0.75]      (0, 0) circle [x radius= 3.35, y radius= 3.35]   ;
\draw [shift={(240,159)}, rotate = 45.73] [color={rgb, 255:red, 0; green, 0; blue, 0 }  ][fill={rgb, 255:red, 0; green, 0; blue, 0 }  ][line width=0.75]      (0, 0) circle [x radius= 3.35, y radius= 3.35]   ;
%Straight Lines [id:da017323852373755333] 
\draw    (201,199) -- (240,239) ;
\draw [shift={(240,239)}, rotate = 45.73] [color={rgb, 255:red, 0; green, 0; blue, 0 }  ][fill={rgb, 255:red, 0; green, 0; blue, 0 }  ][line width=0.75]      (0, 0) circle [x radius= 3.35, y radius= 3.35]   ;
\draw [shift={(201,199)}, rotate = 45.73] [color={rgb, 255:red, 0; green, 0; blue, 0 }  ][fill={rgb, 255:red, 0; green, 0; blue, 0 }  ][line width=0.75]      (0, 0) circle [x radius= 3.35, y radius= 3.35]   ;
%Straight Lines [id:da5169583178251003] 
\draw    (240,239) -- (279,199) ;
\draw [shift={(279,199)}, rotate = 314.27] [color={rgb, 255:red, 0; green, 0; blue, 0 }  ][fill={rgb, 255:red, 0; green, 0; blue, 0 }  ][line width=0.75]      (0, 0) circle [x radius= 3.35, y radius= 3.35]   ;
\draw [shift={(240,239)}, rotate = 314.27] [color={rgb, 255:red, 0; green, 0; blue, 0 }  ][fill={rgb, 255:red, 0; green, 0; blue, 0 }  ][line width=0.75]      (0, 0) circle [x radius= 3.35, y radius= 3.35]   ;
%Straight Lines [id:da03166908404637381] 
\draw    (201,199) -- (240,159) ;
\draw [shift={(240,159)}, rotate = 314.27] [color={rgb, 255:red, 0; green, 0; blue, 0 }  ][fill={rgb, 255:red, 0; green, 0; blue, 0 }  ][line width=0.75]      (0, 0) circle [x radius= 3.35, y radius= 3.35]   ;
\draw [shift={(201,199)}, rotate = 314.27] [color={rgb, 255:red, 0; green, 0; blue, 0 }  ][fill={rgb, 255:red, 0; green, 0; blue, 0 }  ][line width=0.75]      (0, 0) circle [x radius= 3.35, y radius= 3.35]   ;
%Straight Lines [id:da9405246360882455] 
\draw    (171,82) -- (201,199) ;
\draw [shift={(201,199)}, rotate = 75.62] [color={rgb, 255:red, 0; green, 0; blue, 0 }  ][fill={rgb, 255:red, 0; green, 0; blue, 0 }  ][line width=0.75]      (0, 0) circle [x radius= 3.35, y radius= 3.35]   ;
\draw [shift={(171,82)}, rotate = 75.62] [color={rgb, 255:red, 0; green, 0; blue, 0 }  ][fill={rgb, 255:red, 0; green, 0; blue, 0 }  ][line width=0.75]      (0, 0) circle [x radius= 3.35, y radius= 3.35]   ;
%Straight Lines [id:da5873589559482049] 
\draw    (100,159) -- (139,199) ;
\draw [shift={(139,199)}, rotate = 45.73] [color={rgb, 255:red, 0; green, 0; blue, 0 }  ][fill={rgb, 255:red, 0; green, 0; blue, 0 }  ][line width=0.75]      (0, 0) circle [x radius= 3.35, y radius= 3.35]   ;
\draw [shift={(100,159)}, rotate = 45.73] [color={rgb, 255:red, 0; green, 0; blue, 0 }  ][fill={rgb, 255:red, 0; green, 0; blue, 0 }  ][line width=0.75]      (0, 0) circle [x radius= 3.35, y radius= 3.35]   ;
%Straight Lines [id:da754305302058665] 
\draw    (61,199) -- (100,239) ;
\draw [shift={(100,239)}, rotate = 45.73] [color={rgb, 255:red, 0; green, 0; blue, 0 }  ][fill={rgb, 255:red, 0; green, 0; blue, 0 }  ][line width=0.75]      (0, 0) circle [x radius= 3.35, y radius= 3.35]   ;
\draw [shift={(61,199)}, rotate = 45.73] [color={rgb, 255:red, 0; green, 0; blue, 0 }  ][fill={rgb, 255:red, 0; green, 0; blue, 0 }  ][line width=0.75]      (0, 0) circle [x radius= 3.35, y radius= 3.35]   ;
%Straight Lines [id:da12052551359721453] 
\draw    (100,239) -- (139,199) ;
\draw [shift={(139,199)}, rotate = 314.27] [color={rgb, 255:red, 0; green, 0; blue, 0 }  ][fill={rgb, 255:red, 0; green, 0; blue, 0 }  ][line width=0.75]      (0, 0) circle [x radius= 3.35, y radius= 3.35]   ;
\draw [shift={(100,239)}, rotate = 314.27] [color={rgb, 255:red, 0; green, 0; blue, 0 }  ][fill={rgb, 255:red, 0; green, 0; blue, 0 }  ][line width=0.75]      (0, 0) circle [x radius= 3.35, y radius= 3.35]   ;
%Straight Lines [id:da907430436254514] 
\draw    (61,199) -- (100,159) ;
\draw [shift={(100,159)}, rotate = 314.27] [color={rgb, 255:red, 0; green, 0; blue, 0 }  ][fill={rgb, 255:red, 0; green, 0; blue, 0 }  ][line width=0.75]      (0, 0) circle [x radius= 3.35, y radius= 3.35]   ;
\draw [shift={(61,199)}, rotate = 314.27] [color={rgb, 255:red, 0; green, 0; blue, 0 }  ][fill={rgb, 255:red, 0; green, 0; blue, 0 }  ][line width=0.75]      (0, 0) circle [x radius= 3.35, y radius= 3.35]   ;
%Straight Lines [id:da29005989285128164] 
\draw    (139,199) -- (171,82) ;
\draw [shift={(171,82)}, rotate = 285.3] [color={rgb, 255:red, 0; green, 0; blue, 0 }  ][fill={rgb, 255:red, 0; green, 0; blue, 0 }  ][line width=0.75]      (0, 0) circle [x radius= 3.35, y radius= 3.35]   ;
\draw [shift={(139,199)}, rotate = 285.3] [color={rgb, 255:red, 0; green, 0; blue, 0 }  ][fill={rgb, 255:red, 0; green, 0; blue, 0 }  ][line width=0.75]      (0, 0) circle [x radius= 3.35, y radius= 3.35]   ;
%Straight Lines [id:da5874477826233975] 
\draw    (42,70) -- (81,30) ;
\draw [shift={(81,30)}, rotate = 314.27] [color={rgb, 255:red, 0; green, 0; blue, 0 }  ][fill={rgb, 255:red, 0; green, 0; blue, 0 }  ][line width=0.75]      (0, 0) circle [x radius= 3.35, y radius= 3.35]   ;
\draw [shift={(42,70)}, rotate = 314.27] [color={rgb, 255:red, 0; green, 0; blue, 0 }  ][fill={rgb, 255:red, 0; green, 0; blue, 0 }  ][line width=0.75]      (0, 0) circle [x radius= 3.35, y radius= 3.35]   ;
%Straight Lines [id:da8524298411530875] 
\draw    (81,110) -- (120,70) ;
\draw [shift={(120,70)}, rotate = 314.27] [color={rgb, 255:red, 0; green, 0; blue, 0 }  ][fill={rgb, 255:red, 0; green, 0; blue, 0 }  ][line width=0.75]      (0, 0) circle [x radius= 3.35, y radius= 3.35]   ;
\draw [shift={(81,110)}, rotate = 314.27] [color={rgb, 255:red, 0; green, 0; blue, 0 }  ][fill={rgb, 255:red, 0; green, 0; blue, 0 }  ][line width=0.75]      (0, 0) circle [x radius= 3.35, y radius= 3.35]   ;
%Straight Lines [id:da6187555297294258] 
\draw    (218,70) -- (257,30) ;
\draw [shift={(257,30)}, rotate = 314.27] [color={rgb, 255:red, 0; green, 0; blue, 0 }  ][fill={rgb, 255:red, 0; green, 0; blue, 0 }  ][line width=0.75]      (0, 0) circle [x radius= 3.35, y radius= 3.35]   ;
\draw [shift={(218,70)}, rotate = 314.27] [color={rgb, 255:red, 0; green, 0; blue, 0 }  ][fill={rgb, 255:red, 0; green, 0; blue, 0 }  ][line width=0.75]      (0, 0) circle [x radius= 3.35, y radius= 3.35]   ;
%Straight Lines [id:da5356916043754074] 
\draw    (391.27,30) -- (430.27,70) ;
\draw [shift={(430.27,70)}, rotate = 45.73] [color={rgb, 255:red, 0; green, 0; blue, 0 }  ][fill={rgb, 255:red, 0; green, 0; blue, 0 }  ][line width=0.75]      (0, 0) circle [x radius= 3.35, y radius= 3.35]   ;
\draw [shift={(391.27,30)}, rotate = 45.73] [color={rgb, 255:red, 0; green, 0; blue, 0 }  ][fill={rgb, 255:red, 0; green, 0; blue, 0 }  ][line width=0.75]      (0, 0) circle [x radius= 3.35, y radius= 3.35]   ;
%Straight Lines [id:da372300328859956] 
\draw    (352.27,70) -- (391.27,110) ;
\draw [shift={(391.27,110)}, rotate = 45.73] [color={rgb, 255:red, 0; green, 0; blue, 0 }  ][fill={rgb, 255:red, 0; green, 0; blue, 0 }  ][line width=0.75]      (0, 0) circle [x radius= 3.35, y radius= 3.35]   ;
\draw [shift={(352.27,70)}, rotate = 45.73] [color={rgb, 255:red, 0; green, 0; blue, 0 }  ][fill={rgb, 255:red, 0; green, 0; blue, 0 }  ][line width=0.75]      (0, 0) circle [x radius= 3.35, y radius= 3.35]   ;
%Straight Lines [id:da08966109184332827] 
\draw    (567.27,30) -- (606.27,70) ;
\draw [shift={(606.27,70)}, rotate = 45.73] [color={rgb, 255:red, 0; green, 0; blue, 0 }  ][fill={rgb, 255:red, 0; green, 0; blue, 0 }  ][line width=0.75]      (0, 0) circle [x radius= 3.35, y radius= 3.35]   ;
\draw [shift={(567.27,30)}, rotate = 45.73] [color={rgb, 255:red, 0; green, 0; blue, 0 }  ][fill={rgb, 255:red, 0; green, 0; blue, 0 }  ][line width=0.75]      (0, 0) circle [x radius= 3.35, y radius= 3.35]   ;
%Straight Lines [id:da4641171989120647] 
\draw    (528.27,70) -- (567.27,110) ;
\draw [shift={(567.27,110)}, rotate = 45.73] [color={rgb, 255:red, 0; green, 0; blue, 0 }  ][fill={rgb, 255:red, 0; green, 0; blue, 0 }  ][line width=0.75]      (0, 0) circle [x radius= 3.35, y radius= 3.35]   ;
\draw [shift={(528.27,70)}, rotate = 45.73] [color={rgb, 255:red, 0; green, 0; blue, 0 }  ][fill={rgb, 255:red, 0; green, 0; blue, 0 }  ][line width=0.75]      (0, 0) circle [x radius= 3.35, y radius= 3.35]   ;
%Straight Lines [id:da576787823600228] 
\draw    (567.27,110) -- (606.27,70) ;
\draw [shift={(606.27,70)}, rotate = 314.27] [color={rgb, 255:red, 0; green, 0; blue, 0 }  ][fill={rgb, 255:red, 0; green, 0; blue, 0 }  ][line width=0.75]      (0, 0) circle [x radius= 3.35, y radius= 3.35]   ;
\draw [shift={(567.27,110)}, rotate = 314.27] [color={rgb, 255:red, 0; green, 0; blue, 0 }  ][fill={rgb, 255:red, 0; green, 0; blue, 0 }  ][line width=0.75]      (0, 0) circle [x radius= 3.35, y radius= 3.35]   ;
%Straight Lines [id:da9068433888580301] 
\draw    (550.27,159) -- (589.27,199) ;
\draw [shift={(589.27,199)}, rotate = 45.73] [color={rgb, 255:red, 0; green, 0; blue, 0 }  ][fill={rgb, 255:red, 0; green, 0; blue, 0 }  ][line width=0.75]      (0, 0) circle [x radius= 3.35, y radius= 3.35]   ;
\draw [shift={(550.27,159)}, rotate = 45.73] [color={rgb, 255:red, 0; green, 0; blue, 0 }  ][fill={rgb, 255:red, 0; green, 0; blue, 0 }  ][line width=0.75]      (0, 0) circle [x radius= 3.35, y radius= 3.35]   ;
%Straight Lines [id:da08680444614902161] 
\draw    (511.27,199) -- (550.27,239) ;
\draw [shift={(550.27,239)}, rotate = 45.73] [color={rgb, 255:red, 0; green, 0; blue, 0 }  ][fill={rgb, 255:red, 0; green, 0; blue, 0 }  ][line width=0.75]      (0, 0) circle [x radius= 3.35, y radius= 3.35]   ;
\draw [shift={(511.27,199)}, rotate = 45.73] [color={rgb, 255:red, 0; green, 0; blue, 0 }  ][fill={rgb, 255:red, 0; green, 0; blue, 0 }  ][line width=0.75]      (0, 0) circle [x radius= 3.35, y radius= 3.35]   ;
%Straight Lines [id:da23248932026427194] 
\draw    (550.27,239) -- (589.27,199) ;
\draw [shift={(589.27,199)}, rotate = 314.27] [color={rgb, 255:red, 0; green, 0; blue, 0 }  ][fill={rgb, 255:red, 0; green, 0; blue, 0 }  ][line width=0.75]      (0, 0) circle [x radius= 3.35, y radius= 3.35]   ;
\draw [shift={(550.27,239)}, rotate = 314.27] [color={rgb, 255:red, 0; green, 0; blue, 0 }  ][fill={rgb, 255:red, 0; green, 0; blue, 0 }  ][line width=0.75]      (0, 0) circle [x radius= 3.35, y radius= 3.35]   ;
%Straight Lines [id:da7764958047569086] 
\draw    (511.27,199) -- (550.27,159) ;
\draw [shift={(550.27,159)}, rotate = 314.27] [color={rgb, 255:red, 0; green, 0; blue, 0 }  ][fill={rgb, 255:red, 0; green, 0; blue, 0 }  ][line width=0.75]      (0, 0) circle [x radius= 3.35, y radius= 3.35]   ;
\draw [shift={(511.27,199)}, rotate = 314.27] [color={rgb, 255:red, 0; green, 0; blue, 0 }  ][fill={rgb, 255:red, 0; green, 0; blue, 0 }  ][line width=0.75]      (0, 0) circle [x radius= 3.35, y radius= 3.35]   ;
%Straight Lines [id:da7415867152537337] 
\draw    (481.27,82) -- (511.27,199) ;
\draw [shift={(511.27,199)}, rotate = 75.62] [color={rgb, 255:red, 0; green, 0; blue, 0 }  ][fill={rgb, 255:red, 0; green, 0; blue, 0 }  ][line width=0.75]      (0, 0) circle [x radius= 3.35, y radius= 3.35]   ;
\draw [shift={(481.27,82)}, rotate = 75.62] [color={rgb, 255:red, 0; green, 0; blue, 0 }  ][fill={rgb, 255:red, 0; green, 0; blue, 0 }  ][line width=0.75]      (0, 0) circle [x radius= 3.35, y radius= 3.35]   ;
%Straight Lines [id:da03534703800039596] 
\draw    (410.27,159) -- (449.27,199) ;
\draw [shift={(449.27,199)}, rotate = 45.73] [color={rgb, 255:red, 0; green, 0; blue, 0 }  ][fill={rgb, 255:red, 0; green, 0; blue, 0 }  ][line width=0.75]      (0, 0) circle [x radius= 3.35, y radius= 3.35]   ;
\draw [shift={(410.27,159)}, rotate = 45.73] [color={rgb, 255:red, 0; green, 0; blue, 0 }  ][fill={rgb, 255:red, 0; green, 0; blue, 0 }  ][line width=0.75]      (0, 0) circle [x radius= 3.35, y radius= 3.35]   ;
%Straight Lines [id:da24232107906706157] 
\draw    (371.27,199) -- (410.27,239) ;
\draw [shift={(410.27,239)}, rotate = 45.73] [color={rgb, 255:red, 0; green, 0; blue, 0 }  ][fill={rgb, 255:red, 0; green, 0; blue, 0 }  ][line width=0.75]      (0, 0) circle [x radius= 3.35, y radius= 3.35]   ;
\draw [shift={(371.27,199)}, rotate = 45.73] [color={rgb, 255:red, 0; green, 0; blue, 0 }  ][fill={rgb, 255:red, 0; green, 0; blue, 0 }  ][line width=0.75]      (0, 0) circle [x radius= 3.35, y radius= 3.35]   ;
%Straight Lines [id:da10760249782387721] 
\draw    (410.27,239) -- (449.27,199) ;
\draw [shift={(449.27,199)}, rotate = 314.27] [color={rgb, 255:red, 0; green, 0; blue, 0 }  ][fill={rgb, 255:red, 0; green, 0; blue, 0 }  ][line width=0.75]      (0, 0) circle [x radius= 3.35, y radius= 3.35]   ;
\draw [shift={(410.27,239)}, rotate = 314.27] [color={rgb, 255:red, 0; green, 0; blue, 0 }  ][fill={rgb, 255:red, 0; green, 0; blue, 0 }  ][line width=0.75]      (0, 0) circle [x radius= 3.35, y radius= 3.35]   ;
%Straight Lines [id:da38979974682103347] 
\draw    (371.27,199) -- (410.27,159) ;
\draw [shift={(410.27,159)}, rotate = 314.27] [color={rgb, 255:red, 0; green, 0; blue, 0 }  ][fill={rgb, 255:red, 0; green, 0; blue, 0 }  ][line width=0.75]      (0, 0) circle [x radius= 3.35, y radius= 3.35]   ;
\draw [shift={(371.27,199)}, rotate = 314.27] [color={rgb, 255:red, 0; green, 0; blue, 0 }  ][fill={rgb, 255:red, 0; green, 0; blue, 0 }  ][line width=0.75]      (0, 0) circle [x radius= 3.35, y radius= 3.35]   ;
%Straight Lines [id:da7211821977350994] 
\draw    (449.27,199) -- (481.27,82) ;
\draw [shift={(481.27,82)}, rotate = 285.3] [color={rgb, 255:red, 0; green, 0; blue, 0 }  ][fill={rgb, 255:red, 0; green, 0; blue, 0 }  ][line width=0.75]      (0, 0) circle [x radius= 3.35, y radius= 3.35]   ;
\draw [shift={(449.27,199)}, rotate = 285.3] [color={rgb, 255:red, 0; green, 0; blue, 0 }  ][fill={rgb, 255:red, 0; green, 0; blue, 0 }  ][line width=0.75]      (0, 0) circle [x radius= 3.35, y radius= 3.35]   ;
%Straight Lines [id:da9884094178533243] 
\draw    (352.27,70) -- (391.27,30) ;
\draw [shift={(391.27,30)}, rotate = 314.27] [color={rgb, 255:red, 0; green, 0; blue, 0 }  ][fill={rgb, 255:red, 0; green, 0; blue, 0 }  ][line width=0.75]      (0, 0) circle [x radius= 3.35, y radius= 3.35]   ;
\draw [shift={(352.27,70)}, rotate = 314.27] [color={rgb, 255:red, 0; green, 0; blue, 0 }  ][fill={rgb, 255:red, 0; green, 0; blue, 0 }  ][line width=0.75]      (0, 0) circle [x radius= 3.35, y radius= 3.35]   ;
%Straight Lines [id:da3266259534482505] 
\draw    (391.27,110) -- (430.27,70) ;
\draw [shift={(430.27,70)}, rotate = 314.27] [color={rgb, 255:red, 0; green, 0; blue, 0 }  ][fill={rgb, 255:red, 0; green, 0; blue, 0 }  ][line width=0.75]      (0, 0) circle [x radius= 3.35, y radius= 3.35]   ;
\draw [shift={(391.27,110)}, rotate = 314.27] [color={rgb, 255:red, 0; green, 0; blue, 0 }  ][fill={rgb, 255:red, 0; green, 0; blue, 0 }  ][line width=0.75]      (0, 0) circle [x radius= 3.35, y radius= 3.35]   ;
%Straight Lines [id:da1769048589563924] 
\draw    (528.27,70) -- (567.27,30) ;
\draw [shift={(567.27,30)}, rotate = 314.27] [color={rgb, 255:red, 0; green, 0; blue, 0 }  ][fill={rgb, 255:red, 0; green, 0; blue, 0 }  ][line width=0.75]      (0, 0) circle [x radius= 3.35, y radius= 3.35]   ;
\draw [shift={(528.27,70)}, rotate = 314.27] [color={rgb, 255:red, 0; green, 0; blue, 0 }  ][fill={rgb, 255:red, 0; green, 0; blue, 0 }  ][line width=0.75]      (0, 0) circle [x radius= 3.35, y radius= 3.35]   ;

% Text Node
\draw (169,64) node [anchor=north west][inner sep=0.75pt]  [font=\footnotesize] [align=left] {$\displaystyle v_{1}$};
% Text Node
\draw (25,53) node [anchor=north west][inner sep=0.75pt]  [font=\footnotesize] [align=left] {$\displaystyle v_{2}$};
% Text Node
\draw (82,13) node [anchor=north west][inner sep=0.75pt]  [font=\footnotesize] [align=left] {$\displaystyle v_{3}$};
% Text Node
\draw (122,73) node [anchor=north west][inner sep=0.75pt]  [font=\footnotesize] [align=left] {$\displaystyle v_{4}$};
% Text Node
\draw (66,112) node [anchor=north west][inner sep=0.75pt]  [font=\footnotesize] [align=left] {$\displaystyle v_{5}$};
% Text Node
\draw (203,53) node [anchor=north west][inner sep=0.75pt]  [font=\footnotesize] [align=left] {$\displaystyle v_{6}$};
% Text Node
\draw (259,13) node [anchor=north west][inner sep=0.75pt]  [font=\footnotesize] [align=left] {$\displaystyle v_{7}$};
% Text Node
\draw (298,73) node [anchor=north west][inner sep=0.75pt]  [font=\footnotesize] [align=left] {$\displaystyle v_{8}$};
% Text Node
\draw (259,113) node [anchor=north west][inner sep=0.75pt]  [font=\footnotesize] [align=left] {$\displaystyle v_{9}$};
% Text Node
\draw (219.67,44.73) node [anchor=north west][inner sep=0.75pt]  [font=\footnotesize,rotate=-315.65] [align=left] {$\displaystyle ( 12)$};
% Text Node
\draw (269.67,98.73) node [anchor=north west][inner sep=0.75pt]  [font=\footnotesize,rotate=-315.65] [align=left] {$\displaystyle ( 12)$};
% Text Node
\draw (277.48,30.25) node [anchor=north west][inner sep=0.75pt]  [font=\footnotesize,rotate=-44.47] [align=left] {$\displaystyle ( 34)$};
% Text Node
\draw (226.48,82.25) node [anchor=north west][inner sep=0.75pt]  [font=\footnotesize,rotate=-44.47] [align=left] {$\displaystyle ( 34)$};
% Text Node
\draw (71.96,96.27) node [anchor=north west][inner sep=0.75pt]  [font=\footnotesize,rotate=-314.21] [align=left] {$\displaystyle ( 12)( 34)$};
% Text Node
\draw (33.96,55.27) node [anchor=north west][inner sep=0.75pt]  [font=\footnotesize,rotate=-314.21] [align=left] {$\displaystyle ( 12)( 34)$};
% Text Node
\draw (180,202) node [anchor=north west][inner sep=0.75pt]  [font=\footnotesize] [align=left] {$\displaystyle v_{14}$};
% Text Node
\draw (243,146) node [anchor=north west][inner sep=0.75pt]  [font=\footnotesize] [align=left] {$\displaystyle v_{15}$};
% Text Node
\draw (281,202) node [anchor=north west][inner sep=0.75pt]  [font=\footnotesize] [align=left] {$\displaystyle v_{16}$};
% Text Node
\draw (242,242) node [anchor=north west][inner sep=0.75pt]  [font=\footnotesize] [align=left] {$\displaystyle v_{17}$};
% Text Node
\draw (202.67,173.73) node [anchor=north west][inner sep=0.75pt]  [font=\footnotesize,rotate=-315.65] [align=left] {$\displaystyle ( 12)$};
% Text Node
\draw (252.67,227.73) node [anchor=north west][inner sep=0.75pt]  [font=\footnotesize,rotate=-315.65] [align=left] {$\displaystyle ( 12)$};
% Text Node
\draw (260.48,159.25) node [anchor=north west][inner sep=0.75pt]  [font=\footnotesize,rotate=-44.47] [align=left] {$\displaystyle ( 34)$};
% Text Node
\draw (209.48,211.25) node [anchor=north west][inner sep=0.75pt]  [font=\footnotesize,rotate=-44.47] [align=left] {$\displaystyle ( 34)$};
% Text Node
\draw (39,190) node [anchor=north west][inner sep=0.75pt]  [font=\footnotesize] [align=left] {$\displaystyle v_{10}$};
% Text Node
\draw (102,145) node [anchor=north west][inner sep=0.75pt]  [font=\footnotesize] [align=left] {$\displaystyle v_{11}$};
% Text Node
\draw (141,202) node [anchor=north west][inner sep=0.75pt]  [font=\footnotesize] [align=left] {$\displaystyle v_{12}$};
% Text Node
\draw (85,241) node [anchor=north west][inner sep=0.75pt]  [font=\footnotesize] [align=left] {$\displaystyle v_{13}$};
% Text Node
\draw (90.96,225.27) node [anchor=north west][inner sep=0.75pt]  [font=\footnotesize,rotate=-314.21] [align=left] {$\displaystyle ( 12)( 34)$};
% Text Node
\draw (52.96,184.27) node [anchor=north west][inner sep=0.75pt]  [font=\footnotesize,rotate=-314.21] [align=left] {$\displaystyle ( 12)( 34)$};
% Text Node
\draw (466.27,66) node [anchor=north west][inner sep=0.75pt]  [font=\footnotesize] [align=left] {$\displaystyle v_{1}$};
% Text Node
\draw (335.27,53) node [anchor=north west][inner sep=0.75pt]  [font=\footnotesize] [align=left] {$\displaystyle v_{2}$};
% Text Node
\draw (392.27,13) node [anchor=north west][inner sep=0.75pt]  [font=\footnotesize] [align=left] {$\displaystyle v_{3}$};
% Text Node
\draw (432.27,73) node [anchor=north west][inner sep=0.75pt]  [font=\footnotesize] [align=left] {$\displaystyle v_{4}$};
% Text Node
\draw (376.27,112) node [anchor=north west][inner sep=0.75pt]  [font=\footnotesize] [align=left] {$\displaystyle v_{5}$};
% Text Node
\draw (513.27,53) node [anchor=north west][inner sep=0.75pt]  [font=\footnotesize] [align=left] {$\displaystyle v_{6}$};
% Text Node
\draw (569.27,13) node [anchor=north west][inner sep=0.75pt]  [font=\footnotesize] [align=left] {$\displaystyle v_{7}$};
% Text Node
\draw (608.27,73) node [anchor=north west][inner sep=0.75pt]  [font=\footnotesize] [align=left] {$\displaystyle v_{8}$};
% Text Node
\draw (569.27,113) node [anchor=north west][inner sep=0.75pt]  [font=\footnotesize] [align=left] {$\displaystyle v_{9}$};
% Text Node
\draw (529.93,44.73) node [anchor=north west][inner sep=0.75pt]  [font=\footnotesize,rotate=-315.65] [align=left] {$\displaystyle ( 12)$};
% Text Node
\draw (579.93,98.73) node [anchor=north west][inner sep=0.75pt]  [font=\footnotesize,rotate=-315.65] [align=left] {$\displaystyle ( 12)$};
% Text Node
\draw (587.75,30.25) node [anchor=north west][inner sep=0.75pt]  [font=\footnotesize,rotate=-44.47] [align=left] {$\displaystyle ( 34)$};
% Text Node
\draw (536.75,82.25) node [anchor=north west][inner sep=0.75pt]  [font=\footnotesize,rotate=-44.47] [align=left] {$\displaystyle ( 34)$};
% Text Node
\draw (382.23,96.27) node [anchor=north west][inner sep=0.75pt]  [font=\footnotesize,rotate=-314.21] [align=left] {$\displaystyle ( 12)( 34)$};
% Text Node
\draw (344.23,55.27) node [anchor=north west][inner sep=0.75pt]  [font=\footnotesize,rotate=-314.21] [align=left] {$\displaystyle ( 12)( 34)$};
% Text Node
\draw (490.27,202) node [anchor=north west][inner sep=0.75pt]  [font=\footnotesize] [align=left] {$\displaystyle v_{14}$};
% Text Node
\draw (553.27,146) node [anchor=north west][inner sep=0.75pt]  [font=\footnotesize] [align=left] {$\displaystyle v_{15}$};
% Text Node
\draw (591.27,202) node [anchor=north west][inner sep=0.75pt]  [font=\footnotesize] [align=left] {$\displaystyle v_{16}$};
% Text Node
\draw (552.27,242) node [anchor=north west][inner sep=0.75pt]  [font=\footnotesize] [align=left] {$\displaystyle v_{17}$};
% Text Node
\draw (512.93,173.73) node [anchor=north west][inner sep=0.75pt]  [font=\footnotesize,rotate=-315.65] [align=left] {$\displaystyle ( 12)$};
% Text Node
\draw (562.93,227.73) node [anchor=north west][inner sep=0.75pt]  [font=\footnotesize,rotate=-315.65] [align=left] {$\displaystyle ( 12)$};
% Text Node
\draw (570.75,159.25) node [anchor=north west][inner sep=0.75pt]  [font=\footnotesize,rotate=-44.47] [align=left] {$\displaystyle ( 34)$};
% Text Node
\draw (519.75,211.25) node [anchor=north west][inner sep=0.75pt]  [font=\footnotesize,rotate=-44.47] [align=left] {$\displaystyle ( 34)$};
% Text Node
\draw (348.27,189) node [anchor=north west][inner sep=0.75pt]  [font=\footnotesize] [align=left] {$\displaystyle v_{10}$};
% Text Node
\draw (412.27,145) node [anchor=north west][inner sep=0.75pt]  [font=\footnotesize] [align=left] {$\displaystyle v_{11}$};
% Text Node
\draw (451.27,202) node [anchor=north west][inner sep=0.75pt]  [font=\footnotesize] [align=left] {$\displaystyle v_{12}$};
% Text Node
\draw (395.27,241) node [anchor=north west][inner sep=0.75pt]  [font=\footnotesize] [align=left] {$\displaystyle v_{13}$};
% Text Node
\draw (401.23,225.27) node [anchor=north west][inner sep=0.75pt]  [font=\footnotesize,rotate=-314.21] [align=left] {$\displaystyle ( 12)( 34)$};
% Text Node
\draw (363.23,184.27) node [anchor=north west][inner sep=0.75pt]  [font=\footnotesize,rotate=-314.21] [align=left] {$\displaystyle ( 12)( 34)$};

\end{tikzpicture}
    \caption{A pair of $\pi$-cospectral not switching isomorphic $S_4$-gain graphs from Example~\ref{ex:picosp}.}
    \label{fig:picosp}
\end{figure}
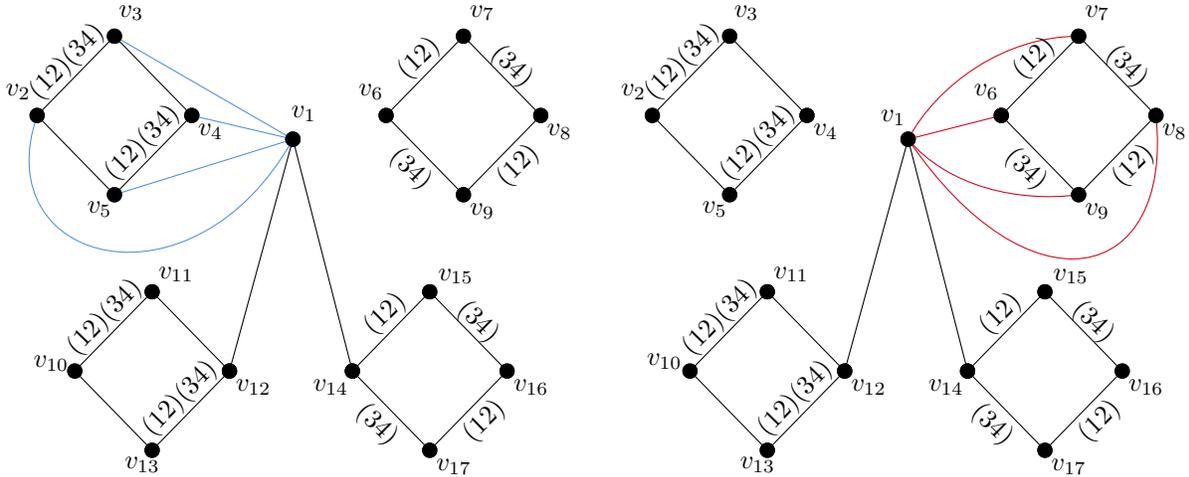

\begin{example}\label{ex:picosp}
%\textcolor{red}{Add example that can be found via WQH but not GM from \cite{CDS2022}}
Consider the $S_4$-gain graph $(\Gamma,\psi)$ and a partition $\alpha$ from Example~\ref{ex:piWQH}. It was already shown that $\alpha$ is a $\pi$-WQH partition where $\pi$ is a unitary representation that maps an element of $S_4$ to a respective permutation matrix of size $4$.
%depicted in the left of Figure~\ref{fig:picosp}; the $\psi$-image of any unlabeled edge is $1_{S_4}$. Let $\alpha$ be a partition $\{C_0,C_1,C_2,C_3,C_4\}$, where $C_0=\{v_1\}$, $C_1=\{v_2,v_3,v_4,v_5\}$, $C_2=\{v_6,v_7,v_8,v_9\}$, $C_3=\{v_{10},v_{11},v_{12},v_{13}\}$, and $C_4=\{v_{14},v_{15},v_{16},v_{17}\}$. It is not a $G$-WQH partition, since $\Psi_1(v)=1_{S_4}+(12)(34)\neq (12)+(34)=\Psi_2(v')$ for $v\in C_1$ and $v'\in C_2$. However, by choosing a unitary representation $\pi$ that sends a permutation from $S_4$ to a respective $4\times 4$ permutation matrix, we get $\pi(\Psi_1(v))=\pi(\Psi_2(v'))$, and in general, all conditions of a $\pi$-WQH partition hold for $\alpha$. 
The graph $(\Gamma^\alpha,\psi^\alpha)$ (on the right in Figure~\ref{fig:picosp}) is not switching isomorphic to $(\Gamma,\psi)$. Additionally, there is no partition that both satisfies the conditions for $\pi$-GM partition for some $\pi$ (defined in \cite{CDS2022}) and produces a graph that is switching isomorphic to $(\Gamma^\alpha,\psi^\alpha)$.
\end{example}

%%%%%%%%%%%%%%%%%%%%%%%%%%%%%%%%%%%%%%%%%%%%%%%%%
\section{Concluding remarks}
%%%%%%%%%%%%%%%%%%%%%%%%%%%%%%%%%%%%%%%%%%%%%%%%%

The generalization of the spectral graph theory to gain graphs is far from trivial. Switching methods for gain graphs were previously described in the literature for some particular cases. In \cite[Section 4]{Belardo1}, WQH-switching (referred to as Modified Godsil-McKay switching) was generalized to signed graphs. In \cite[Section 3]{Belardo2}, the same switching was described for complex unit gain graphs. In both cases, adaptations of WQH-switching were used to construct cospectral graphs with respect to the identical representation. For a general group $G\not\subseteq\C$, an approach based on group representations was introduced in \cite{CDD2021} in order to discuss the spectrum of a $G$-gain graph with respect to any representation $\pi$. In \cite{CD2022}, the notion of $G$-cospectrality independent of the choice of $\pi$ was first introduced, and in \cite{CDS2022}, generalizations of GM-switching with respect to both $G$- and $\pi$-cospectrality were described.
To the authors' knowledge, only the work by Cavaleri, Donno and Spessato \cite{CD2022,CDS2022} contributed before to the investigation of the cospectrality in the context of gain graphs over a general group $G$.
% there is also a nice paper on the line graphs of gain graphs by Cavaleri Donno and D'Angeli. So we should restrict the above comment to cospectrality.

%%%%%%%%%%%%%%%%%%%%%%%%%%%%%%%%%%%%%%%%%%%%%%%%%
\subsection*{Acknowledgements}
%%%%%%%%%%%%%%%%%%%%%%%%%%%%%%%%%%%%%%%%%%%%%%%%%
Aida Abiad is partially supported by FWO (Research Foundation Flanders) via the grant 1285921N. This research is supported by NWO (Dutch Research Council) via an ENW-KLEIN-1 project (OCENW.KLEIN.475).

%%%%%%%%%%%%%%%%%%%%%%%%%%%%%%%%%%%%%%%%%%%%%%%%%

\end{document}